\documentclass{amsart}

\usepackage{amsmath,amsfonts,amsthm,amstext,amssymb,amscd,amsbsy}
\usepackage[latin1]{inputenc}
\usepackage[T1]{fontenc}
\usepackage{indentfirst} 
\usepackage[dvips]{color}
\usepackage{epsfig}
\usepackage{color}
\usepackage{graphics}

\newtheorem{definicao}{Definition}%[chapter]

\newtheorem{lema}[definicao]{Lemma}%[section]
\newtheorem{proposicao}[definicao]{Proposition}%[section]
\newtheorem{obs}[definicao]{Remark}%[section]
\newtheorem{teorema}[definicao]{Theorem}%[section]
\newtheorem{corolario}[definicao]{Corollary}%[section]
\newcommand{\rn}[1]{\mathbb{R}^{#1}}

\newcommand{\zz}{\mathbb{Z}_{2}}
\newcommand{\dd}{\mathbb{D}_{4}}

\def\bq{\begin{equation}}
\def\eq{\end{equation}}

\begin{document}

%%%%

\title[Reversible-equivariant systems]{Reversible-equivariant systems and matricial equations}

\author[R.M. Martins and M.A. Teixeira]
{Ricardo Miranda Martins and Marco Antonio Teixeira}

\address{Department of Mathematics, Institute of Mathematics, Statistics and Scientific Computing. P.O. Box 6065, University of Campinas - UNICAMP. Campinas, Brazil.
}

\email[R.M. Martins]{rmiranda@ime.unicamp.br}
\email[M.A. Teixeira]{teixeira@ime.unicamp.br}

%\address{$^3$ Departamento de Matematica, Universidade
%Estadual de Campi\-nas, Caixa Postal 6065, 13083--859, Campinas,
%SP, Brazil\bigskip}

%\email{teixeira@ime.unicamp.br}

\thanks{$^*$ This research was partially supported by CNPq Brazil process No 134619/2006-4 (Martins) and by FAPESP Brazil projects numbers 2007/05215-4 (Martins) and 2007/06896-5 (Teixeira).}

%34C20=normal forms
%37C80=equivariant DS
%15A24=matrix equations and IDs


\date{}
\dedicatory{}

\begin{abstract}
This paper uses tools in group theory and symbolic computing to give a classification of the representations of finite groups with order lower than 9 that can be derived from the study of local reversible-equivariant vector fields in $\rn{4}$. The results are obtained by solving algebraically matricial equations. In particular, we exhibit the involutions used in a local study of reversible-equivariant vector fields. Based on such approach we present, for each element in this class, a simplified Belitiskii normal form.

\end{abstract}

\maketitle

%%%

%\tableofcontents

\section{Introduction}

The presence of involutory symmetries and involutory reversing symmetries is very common in physical systems, for example, in classical mechanics, quantum mechanics and thermodynamic (see \cite{lambsurvey}). The theory of ordinary differential equations with symmetry dates back from 1915 with the work of Birkhoff (\cite{birk1915}). Birkhoff realized a special property of its model: the existence of a involutive map $R$ such that the system was symmetric with respect to the set of fixed points of $R$. Since then, the work on differential equations with symmetries stay restricted to hamiltonian equations. Only in 1976, Devaney developed a theory for reversible dynamical systems (\cite{devaney}).

In this paper, involutory symmetries and involutory reversing symmetries are considered within a unified approach. We study some possible linearizations for symmetries and reversing symmetries, around a fixed point, and employ this to simplify normal forms for a class of vector fields.

In particular, using tools from group theory and symbolic computing, we exhibit the involutions used in a local study of reversible-equivariant vector fields. Based on such approach we present, for each element in this class, a simplified Belitiskii normal form.

An important point to mention is that any map possessing an involutory reversing symmetry is the composition of two involutions, as was found by Birkhoff. It is worth to point out that properties of reversing symmetry groups are a powerful tool to study local bifurcation theory in presence of symmetries. Refer to \cite{bi}, where many useful informations on involutions are provided.

\section{Statement of main results}

Let $\mathfrak{X}_0(\rn{4})$ denote the set of all germs of $C^\infty$ vector fields in $\rn{4}$ with a isolated singularity at origin. Define \begin{equation}\label{AA}
A(\alpha,\beta)=\left(\begin{array}{cccc}
0&-\alpha&0&0\\
\alpha&0&0&0\\
0&0&0&-\beta\\
0&0&\beta&0
\end{array}
\right), \ \alpha,\beta\in\mathbb{R}, \ \alpha\beta\neq 0
\end{equation}
 and
\[\mathfrak{X}_0^{(\alpha,\beta)}(\rn{4})=\{X\in \mathfrak{X}_0(\rn{4}); \ DX(0)=A(\alpha,\beta)\}.\]

Given a group $G$ of involutive diffeomorphisms $\rn{4},0\rightarrow\rn{4},0$ and a {group homomorphism} $\rho:G\rightarrow\{- 1,1\}$, we say that $X\in\mathfrak{X}_0(\rn{4})$ is $G$-reversible-equivariant if, for each $\phi\in G$, \[D\phi(x)X(x)=\rho(\phi)X(\phi(x)).\] If $K\subseteq G$ is such that $\rho(K)= 1$ we say that $X$ is $K$-equivariant. If $K\subset G$ is such that $\rho(K)=-1$, we say that $X$ is $K$-reversible. It is clear that if $X$ is $\phi_1$-reversible and $\phi_2$-reversible, then $X$ is also $\phi_1\phi_2$-equivariant. It is usual to denote $G_+=\{\phi\in G; \rho(\phi)=1\}$ and $G_-=\{\phi\in G; \rho(\phi)=-1\}$. Note that $G_+$ is a subgroup of $G$, but $G_-$ is not.

If $X\in\mathfrak{X}_0(\rn{4})$ is $\varphi$-reversible (resp. $\phi$-equivariant) and $\gamma(t)$ is a solution of \bq\label{edo}\dot x=X(x)\eq with $\gamma(0)=x_0$, then $\varphi\gamma(-t)$ (resp. $\phi\gamma(t)$) is also a solution for \eqref{edo}. In particular, if $X$ is a $\phi$-reversible (or $\phi$-equivariant) vector field, the phase portrait of $X$ is symmetric with respect to the subspace $Fix(\phi)$.

A compendium containing work for reversible-equivariant vector fields is described in \cite{patty}, \cite{lambsurvey}, \cite{devaney} and references therein.

In this paper, we shall restrict our study to the case $G$ finite and generated by two involutions, $G=\langle \varphi,\psi\rangle$, in such a way that when $X$ is a $G$-reversible-equivariant vector field, then $X$ is reversible with respect to both $\varphi$ and $\psi$. In this case, by a basic group theory argument, one can prove that there is $n\geq 2$ such that $G\cong \mathbb D_n$.

Our aim is to provide an analogous form of the theorem below to the $G$-reversible-equivariant case:
\begin{teorema}\label{thm1}Let $X\in \mathfrak{X}_0(\rn{2n})$ be a $\varphi$-reversible vector field, where $\varphi:\rn{2n},0\rightarrow\rn{2n},0$ is a $C^\infty$ involution with $\dim Fix(\varphi)=n$ as a local submanifold, and let $R_0:\rn{2n}\rightarrow\rn{2n}$ be any linear involution with $\dim Fix(R_0)=n$. Then there exists a change of coordinates $h:\rn{2n},0\rightarrow\rn{2n},0$ (depending on $R_0$) such that $h_*X$ is $R_0$-reversible.
\end{teorema}

The proof of Theorem \ref{thm1} is straightforward: $\varphi$ is locally conjugated to $D\varphi(0)$ by the change of coordinates $Id+d\varphi(0)\varphi$; now, $D\varphi(0)$ and $R_0$ are linearly conjugated (by $P$, say), as they are linear involutions with $\dim Fix(D\varphi(0))=\dim Fix(R_0)$. Now take $h=P\circ(Id+d\varphi(0)\varphi)$.

Theorem \ref{thm1} is very useful when when one works locally with reversible vector fields. See for example in \cite{tei1} and \cite{tei2}, as it allowed to always fix the involution as the following: \bq\label{r0} R_0(x_1,\ldots,x_{2n})=(x_1,-x_2,\ldots,x_{2n-1},-x_{2n}).\eq

%In the reversible-equivariant context, however, it is not possible to fix the involution.

\begin{definicao}Given a finitely generated group $G=\langle g_1,\ldots,g_l\rangle$, a representation $\rho:G\rightarrow M_{n\times n}(\mathbb{R})$ and a vector field $X\in\mathfrak{X}_0(\rn{n})$, we say that the representation $\rho$ is $(X,G)$-compatible if $\rho(g_j)X(x)=-X(\rho(g_j))$, for all $j=1,\ldots,l$.
\end{definicao}

We prove the following:\\

\noindent{\bf Theorem A:} Given $X\in\mathfrak{X}_0^{(\alpha,\beta)}(\rn{4})$, we present all the $X_{\mathbb D_{n}}$-compatible $4$-dimensional representations of $X$, for $n=2,3,4$.\\

As an application of Theorem A, we obtain the following result.\\

\noindent{\bf Theorem B:} The Belitskii normal form for $\mathbb D_{n}$-reversible-equivariant vector fields ($n=2,3,4$) in $\mathfrak{X}_0^{(\alpha,\beta)}(\rn{4})$ is exhibited.\\

For further details on normal form theory, see \cite{Be} and \cite{bruno}.

This paper is organized as follows. In Section 3 we set the problem and we reduce it to a system of matricial equations. In Section 4 we prove Theorem A and in Section 5, we prove Theorem B.

\section{Setting the problem}\label{secao2}

Consider $X\in\mathfrak{X}_0^{(\alpha,\beta)}(\rn{4})$ for $\alpha\beta\neq 0$. Denote $A=DX(0)$. Let $\varphi,\psi:\rn{4},0\rightarrow\rn{4}$ be involutions with $\dim Fix(\varphi)=\dim Fix(\psi)=2$ and suppose that $X$ is $\langle \varphi,\psi\rangle$-reversible-equivariant.

Next result will be useful in the sequel.

\begin{teorema}[Montgomery-Bochner, \cite{MB1},\cite{MB2},\cite{MB3}]\label{cmb}Let $G$ be a compact group of $C^k$ diffeomorphisms defined on a $C^{k\geq 1}$ manifold $\mathcal{M}$. Suppose that all diffeomorphisms in $G$ have a common fixed point, say $x_0$. Then, there exists a $C^k$ coordinate system $h$ around $x_0$ such that all diffeomorphisms in $G$ are linear with respect to $h$.
\end{teorema}

Putting $G=\langle \varphi,\psi\rangle$, as $\varphi(0)=0$ and $\psi(0)=0$, Theorem \ref{cmb} says that there exists a coordinate system $h$ around $0$ such that $\widetilde \varphi=h^{-1}\varphi h$ and $\widetilde \psi=h^{-1}\psi h$ are linear involutions. Now consider $\widetilde X$ as $X$ is this new system of coordinates. Then $\widetilde X$ is $\langle\widetilde \varphi,\widetilde \psi\rangle$-reversible-equivariant.

Now choose any involution $R_0:\rn{4}\rightarrow\rn{4}$ with $\dim Fix(R_0)=2$. As $R_0$ and $\widetilde\varphi$ are linearly conjugated, we can pass to a new system of coordinates such that $\widetilde X$ is $\langle R_0,{\widetilde{\psi}_0}\rangle$-reversible-equivariant for some ${\widetilde{\psi}_0}$. However, it is not possible to choose a priori a good (linear) representative for the second involution, ${\widetilde{\psi}_0}$.

In other words, it is not possible to produce a analog version of Theorem \ref{thm1} for reversible-equivariant vector fields. We shall take into account all the possible choices for the second involution.\\

\noindent {\bf Problem A:} Let $G=\langle \varphi,\psi\rangle$ be a group generated by involutive diffeomorphisms, and $X\in\mathfrak{X}_0(\rn{2n})$ be a $G$-reversible-equivariant vector field. Find all of the  $(X,G)$-compatible representations $\rho$ with $\rho(\varphi)=R_0$, $R_0$ given by \eqref{r0}.\\

To solve Problem A, we have to determine all the linear involutions $S$ such that $\langle R_0,S\rangle\cong G$ and $SDX(0)+DX(0)S=0$ (this last relation is the compatibility condition for the linear part of $X$).

\section{Proof of Theorem A}

In this section we prove Theorem A. Recall that the list of groups to be considered is: (a) $\mathbb{Z}_2\times\mathbb{Z}_2$, (b) $\mathbb{D}_3$ and (c) $\mathbb{D}_4$. 

\subsection{Case $\zz\times\zz$}\label{secao31}

Fix the matrix
\begin{equation}
R_0=\left(\begin{array}{cccc}
1&0&0&0\\
0&-1&0&0\\
0&0&1&0\\
0&0&0&-1
\end{array}
\right).
\end{equation}

Note that $R_0^2=Id$ and $R_0A=-AR_0$. We need to determine all possible involutive matrices $S\in\rn{4\times 4}$ such that \[SA=-AS\] and \[\langle R_0,S\rangle\cong \mathbb{Z}_2\times\mathbb{Z}_2.\]

Note that the relation $\langle R_0,S\rangle\cong \mathbb{Z}_2\times\mathbb{Z}_2$ is equivalent to $R_0S=SR_0$ and $S^2=Id$.

Put 
\begin{equation}
S=\left(\begin{array}{cccc}
a_1&b_1&c_1&d_1\\
a_2&b_2&c_2&d_2\\
a_3&b_3&c_3&d_3\\
a_4&b_4&c_4&d_4
\end{array}
\right).
\end{equation}

The relations $SA=-AS$, $S^2=Id$ and $R_0S=SR_0$ are represented by the following systems of polynomial equations:
\begin{equation}\label{monstro}
\left\{
\begin{array}{lcl}
a_1^2-1+c_1a_3&=&0\\
a_1c_1+c_1c_3&=&0\\
b_2d_2+d_2d_4&=&0\\
a_3a_1+c_3a_3&=&0\\
b_4b_2+d_4b_4&=&0\\
b_2^2+d_2b_4-1&=&0\\
c_1a_3+c_3^2-1&=&0\\
d_2b_4+d_4^2-1&=&0
\end{array}
\right. \ \ \ \ \ \ \ \ \ \ \ \ \ \
\left\{
\begin{array}{lcl}
-a_1\alpha-b_2\alpha&=&0\\
-c_1\beta-\alpha d_2&=&0\\
b_2\alpha+a_1\alpha&=&0\\
d_2\beta+\alpha c_1&=&0\\
-a_3\alpha-\beta b_4&=&0\\
-c_3\beta-d_4\beta&=&0\\
b_4\alpha+\beta a_3&=&0\\
d_4\beta+c_3\beta&=&0
\end{array}
\right.
\end{equation}

\begin{lema}System \eqref{monstro} has 4 solutions:
\[
S_1=\left(
\begin{array}{cccc}
-1& 0& 0& 0\\
0& 1& 0& 0\\
0& 0& -1& 0\\
0& 0& 0& 1
\end{array}
\right), \ \ S_2=\left(
\begin{array}{cccc}
-1& 0& 0& 0\\
0& 1& 0& 0\\
0& 0& 1& 0\\
0& 0& 0& -1
\end{array}
\right)
\]
\[
S_3=\left(
\begin{array}{cccc}
1& 0& 0& 0\\
0& -1& 0& 0\\
0& 0& -1& 0\\
0& 0& 0& 1
\end{array}
\right), \ \ S_4=\left(
\begin{array}{cccc}
1& 0& 0& 0\\
0& -1& 0& 0\\
0& 0& 1& 0\\
0& 0& 0& -1
\end{array}
\right)
\]
\end{lema}
\begin{proof}This can be done in Maple 12 by means of the \texttt{Reduce} function from the Groebner package and the usual Maple's \texttt{solve} function. We remark that the solution $S_4$ is degenerate, i.e., $S_4=R_0$. Moreover, we remark that the above representations of $\zz\times\zz$ are not equivalent.
\end{proof}

Now we state the main result for $\mathbb{Z}_2\times\mathbb{Z}_2$-reversible vector fields. With the notation of Section \eqref{secao2}, it assures that the linear involutions $S_j$ are the unique possibilities for $\widetilde{\psi_0}$.

\begin{teorema}\label{z2z2}Let $\Omega_{\zz\times\zz}\subset\mathfrak{X}_0^{(\alpha,\beta)}(\rn{4})$ be the set of $\zz\times\zz$-reversible-equivariant vector fields $X\in\mathfrak{X}_0^{(\alpha,\beta)}(\rn{4})$. Then $\Omega=\Omega_1\cup\Omega_2\cup\Omega_3$, where $X\in\Omega_j$ if $X$ is $(R_0,S_j)$-reversible-equivariant in some coordinate system around the origin.
\end{teorema}
\begin{proof}The proof is straightforward and it will be omitted.
\end{proof}

Now let us give a characterization of the vector fields which are $(R_0,S_j)$-reversible. Let us fix \begin{equation}\label{campox}X(x)=A(\alpha,\beta)x+(f_1(x),f_2(x),f_3(x),f_4(x))^T,\end{equation} \noindent with $x\equiv(x_1,x_2,y_1,y_2)$. The proof of next results will be omitted.

\begin{corolario}\label{1stcor}The vector field \eqref{campox} is $(R_0,S_1)$-reversible if and only if the functions $f_j$ satisfies
\begin{equation*}\left\{
\begin{array}{lcllcl}
f_1(x)&=&-f_1(x_1,-x_2,y_1,-y_2)&=&f_1(-x_1,x_2,-y_1,y_2)\\
f_2(x)&=&f_2(x_1,-x_2,y_1,-y_2)&=&-f_2(-x_1,x_2,-y_1,y_2)\\
f_3(x)&=&-f_3(x_1,-x_2,y_1,-y_2)&=&f_3(-x_1,x_2,-y_1,y_2)\\
f_4(x)&=&f_4(x_1,-x_2,y_1,-y_2)&=&-f_4(-x_1,x_2,-y_1,y_2).
\end{array}
\right.
\end{equation*}

\noindent In particular, $f_{1,3}(x_1,0,y_1,0)\equiv 0$ and $f_{2,4}(0,x_2,0,y_2)\equiv 0$.

\end{corolario}

\begin{corolario}\label{2ndcor}The vector field \eqref{campox} is $(R_0,S_2)$-reversible if and only if the functions $f_j$ satisfy
\begin{equation}\left\{
\begin{array}{lcllcl}
f_1(x)&=&-f_1(x_1,-x_2,y_1,-y_2)&=&f_1(-x_1,x_2,y_1,-y_2)\\
f_2(x)&=&f_2(x_1,-x_2,y_1,-y_2)&=&-f_2(-x_1,x_2,y_1,-y_2)\\
f_3(x)&=&-f_3(x_1,-x_2,y_1,-y_2)&=&-f_3(-x_1,x_2,y_1,-y_2)\\
f_4(x)&=&f_4(x_1,-x_2,y_1,-y_2)&=&f_4(-x_1,x_2,y_1,-y_2).
\end{array}
\right.
\end{equation}

\noindent In particular, $f_{1,3}(x_1,0,y_1,0)\equiv 0$ and $f_{2,3}(0,x_2,y_1,0)\equiv 0$.

\end{corolario}
\begin{corolario}\label{3rdcor}The vector field \eqref{campox} is $(R_0,S_3)$-reversible if and only if the functions $f_j$ satisfy
\begin{equation}\left\{
\begin{array}{lcllcl}
f_1(x)&=&-f_1(x_1,-x_2,y_1,-y_2)&=&-f_1(x_1,-x_2,-y_1,y_2)\\
f_2(x)&=&f_2(x_1,-x_2,y_1,-y_2)&=&f_2(x_1,-x_2,-y_1,y_2)\\
f_3(x)&=&-f_3(x_1,-x_2,y_1,-y_2)&=&f_3(x_1,-x_2,-y_1,y_2)\\
f_4(x)&=&f_4(x_1,-x_2,y_1,-y_2)&=&-f_4(x_1,-x_2,-y_1,y_2).
\end{array}
\right.
\end{equation}

\noindent In particular, $f_{1,3}(x_1,0,y_1,0)\equiv 0$ and $f_{1,4}(x_1,0,0,y_2)\equiv 0$.

\end{corolario}

\subsection{Case $\mathbb{D}_3$}

As above we fix the matrix
\begin{equation}
R_0=\left(\begin{array}{cccc}
1&0&0&0\\
0&-1&0&0\\
0&0&1&0\\
0&0&0&-1
\end{array}
\right).
\end{equation}

Now we need to determine all possible involutive matrices $S\in\rn{4\times 4}$ such that \[SA=-AS\] and \[\langle R_0,S\rangle\cong \mathbb{D}_3.\]

Considering again
\begin{equation}
S=\left(\begin{array}{cccc}
a_1&b_1&c_1&d_1\\
a_2&b_2&c_2&d_2\\
a_3&b_3&c_3&d_3\\
a_4&b_4&c_4&d_4
\end{array}
\right),
\end{equation}

\noindent the equations $SA+AS=0$, $S^2-Id=0$ and $(R_0S)^3-Id=0$ are equivalent to a huge system of equations. Its expression will be not presented.

\begin{lema}\label{solmonstro}The system generated by the above conditions has the following non degenerate solutions:
\begin{equation*}
S_1=\left(
\begin{array}{cccc}
-\dfrac{1}{2} & \dfrac{\sqrt{3}}{2} & 0 & 0\\
\dfrac{\sqrt{3}}{2}& \dfrac{1}{2}& 0& 0\\
0& 0& -\dfrac{1}{2}& \dfrac{\sqrt{3}}{2}\\
0& 0& \dfrac{\sqrt{3}}{2}& \dfrac{1}{2}
\end{array}
\right), \ S_2=\left(
\begin{array}{cccc}
-\dfrac{1}{2}& \dfrac{\sqrt{3}}{2}& 0& 0\\
\dfrac{\sqrt{3}}{2}& \dfrac{1}{2}& 0& 0\\
0& 0& 1& 0\\
0& 0& 0& -1
\end{array}
\right), \ S_3=\left(
\begin{array}{cccc}
1& 0& 0& 0\\
0& -1& 0& 0\\
0& 0& -\dfrac{1}{2}& \dfrac{\sqrt{3}}{2}\\
0& 0& \dfrac{\sqrt{3}}{2}& \dfrac{1}{2}
\end{array}
\right).
\end{equation*}

\end{lema}
\begin{proof}Again, the proof can be done in Maple 12 using the \texttt{Reduce} function from the Groebner package and the usual Maple's \texttt{solve} function.
\end{proof}

At this point, we can state the following:

\begin{teorema}\label{d3naopode}Let $\Omega_{\mathbb D_3}\subset\mathfrak{X}_0^{(\alpha,\beta)}(\rn{4})$ be the set of $\mathbb D_3$-reversible-equivariant vector fields $X\in\mathfrak{X}_0^{(\alpha,\beta)}(\rn{4})$. Then $\Omega=\Omega_1\cup\Omega_2\cup\Omega_3$, where $X\in\Omega_j$ if $X$ is $(R_0,S_j)$-reversible-equivariant in some coordinate system around the origin.
\end{teorema}
\begin{proof}This proof is very similar to that of Theorem \ref{z2z2}.
\end{proof}

The next section deals with the characterization of the $\dd$-reversible vector fields. The analysis of the $\mathbb D_{3}$-reversible case will be omitted since it is very similar to the $\dd$-reversible case and this last case is more interesting (there are more representations).

\subsection{Case $\dd$}\label{secaod4}

Fix the matrix
\begin{equation}
R_0=\left(\begin{array}{cccc}
1&0&0&0\\
0&-1&0&0\\
0&0&1&0\\
0&0&0&-1
\end{array}
\right).
\end{equation}

Again, our aim is to determine all of the possible involutive matrices $S\in\rn{4\times 4}$ such that \[SA=-AS\] and \[\langle R_0,S\rangle\cong \mathbb{D}_4.\]

Considering again
\begin{equation}
S=\left(\begin{array}{cccc}
a_1&b_1&c_1&d_1\\
a_2&b_2&c_2&d_2\\
a_3&b_3&c_3&d_3\\
a_4&b_4&c_4&d_4
\end{array}
\right),
\end{equation}

\noindent the equations $SA+AS=0$, $S^2-Id=0$ and $(R_0S)^4-Id=0$\label{ddd} are represented by a huge system (see the Appendix) having $12$ non degenerate solutions, arranged in the following way:

\scriptsize
\[\label{xii}\Xi_1=\left\{\left(
\begin{array}{cccc}
0&-1&0&0\\
-1&0&0&0\\
0&0&1&0\\
0&0&0&-1
\end{array}
\right), \left(
\begin{array}{cccc}
0&1&0&0\\
1&0&0&0\\
0&0&1&0\\
0&0&0&-1
\end{array}
\right)\right\}, \ \Xi_2=\left\{\left(
\begin{array}{cccc}
-1&0&0&0\\
0&1&0&0\\
0&0&0&1\\
0&0&1&0
\end{array}
\right), \left(
\begin{array}{cccc}
-1&0&0&0\\
0&1&0&0\\
0&0&0&-1\\
0&0&-1&0
\end{array}
\right)\right\}
\]
\[
\Xi_3=\left\{\left(
\begin{array}{cccc}
1&0&0&0\\
0&-1&0&0\\
0&0&0&1\\
0&0&1&0
\end{array}
\right), \left(
\begin{array}{cccc}
1&0&0&0\\
0&-1&0&0\\
0&0&0&-1\\
0&0&-1&0
\end{array}
\right)\right\}, \ \Xi_4=\left\{\left(
\begin{array}{cccc}
0&1&0&0\\
1&0&0&0\\
0&0&0&1\\
0&0&1&0
\end{array}
\right), \left(
\begin{array}{cccc}
0&-1&0&0\\
-1&0&0&0\\
0&0&0&-1\\
0&0&-1&0
\end{array}
\right)\right\}
\]
\[\Xi_5=\left\{\left(
\begin{array}{cccc}
0&1&0&0\\
1&0&0&0\\
0&0&-1&0\\
0&0&0&1
\end{array}
\right), \left(
\begin{array}{cccc}
0&-1&0&0\\
-1&0&0&0\\
0&0&-1&0\\
0&0&0&1\end{array}
\right)\right\}, \ \Xi_6=\left\{\left(
\begin{array}{cccc}
0&1&0&0\\
1&0&0&0\\
0&0&0&-1\\
0&0&-1&0
\end{array}
\right), \left(
\begin{array}{cccc}
0&-1&0&0\\
-1&0&0&0\\
0&0&0&1\\
0&0&1&0\end{array}
\right)\right\}
\]
\normalsize

Recall that the above arrangement has obeyed the rule:

\begin{lema}\label{si}$S_i,S_j\in\Xi_k\Leftrightarrow \langle R_0,S_i\rangle=\langle R_0,S_j\rangle$.
\end{lema}

For each $i\in\{1,\ldots, 6\}$, denote by $S_i$ one of the elements of $\Xi_i$. The proof of the next result follows immediately from the above lemmas.

\begin{teorema}\label{d4d4}Let $\Omega_{\dd}\subset\mathfrak{X}_0^{(\alpha,\beta)}(\rn{4})$ be the set of $\dd$-reversible-equivariant vector fields $X\in\mathfrak{X}_0^{(\alpha,\beta)}(\rn{4})$. Then $\Omega=\Omega_1\cup\Omega_2\cup\ldots\cup\Omega_6$, where $X\in\Omega_j$ if $X$ is $(R_0,S_j)$-reversible-equivariant in some coordinate system around the origin.
\end{teorema}
\begin{proof}This proof is very similar to that of Theorem \ref{z2z2}. It will be omitted. 
\end{proof}

Now we present some results in the sense of Corollary \ref{1stcor} applied to $\dd$-reversible vector fields. We will just work with some linearized groups; the other cases are similar.

Let us fix again \begin{equation}\label{campoxx}X(x)=A(\alpha,\beta)x+(g_1(x),g_2(x),g_3(x),g_4(x))^T,\end{equation} \noindent with $x\equiv(x_1,x_2,y_1,y_2)$. Keeping the same notation of Section \ref{secao31}, we have now $\langle R_0,S_j\rangle\cong \dd$.
\begin{corolario}The vector field \eqref{campoxx} is $(R_0,S_1)$-reversible if and only if the functions $g_j$ satisfy
\begin{equation}\label{bran}\left\{
\begin{array}{lcllcl}
g_1(x)&=&-g_1(x_1,-x_2,y_1,-y_2)&=&-g_2(x_2,x_1,y_1,-y_2)\\
g_2(x)&=&g_2(x_1,-x_2,y_1,-y_2)&=&-g_1(x_2,x_1,y_1,-y_2)\\
g_3(x)&=&-g_3(x_1,-x_2,y_1,-y_2)&=&-g_3(x_2,x_1,y_1,-y_2)\\
g_4(x)&=&g_4(x_1,-x_2,y_1,-y_2)&=&g_4(x_2,x_1,y_1,-y_2)
\end{array}
\right.
\end{equation}

\noindent In particular $g_{1,3}(x_1,0,y_1,0)\equiv 0$ and $g_{2,4}(0,x_2,0,y_2)\equiv 0$.
\end{corolario}

\section{Applications to normal forms (Proof of Theorem B)}\label{ultimasecao}

Let $X\in\mathfrak{X}_0^{(\alpha,\beta)}(\rn{4})$ be a $\dd$-reversible vector field and $X^N$ be its reversible-equivariant Belitskii normal form.

To compute the expression of $X^N$, we have to consider the following possibilities of the parameter $\lambda=\alpha\beta^{-1}$:

\noindent (i) $\lambda\notin\mathbb{Q}$,\\
\noindent(ii) $\lambda=1$,\\
\noindent(iii) $\lambda=pq^{-1}$, with $p,q$ integers with $(p,q)=1$.

In the case (i), one can show that the normal forms for the reversible and reversible-equivariant cases are essencially the same. This means that any reversible field with such linear approximation is automatically reversible-equivariant. In view of this, case (i)\label{opcoes} is not interesting, and its analysis will be omitted. We just observe that case (ii) will not be discussed here because of its deep degeneracy, as the range of its homological operator \[L_{A(\alpha,\alpha)}:H^k\rightarrow H^k\] is a very low dimensional subspace of $H^k$. 

Our goal is to focus on the case (iii). Put $\alpha=p$ and $\beta=q$, with $p,q\in\mathbb{Z}$ and $(p,q)=1$. How to compute a normal form which applies for all $\dd$-reversible vector fields, without choosing specific involutions?

According to the results in the last section, it suffices to show that $X^N$ satisfies \[R_0(X^N(x))=-X^N(R_0(x))\] and \[S_j(X^N(x))=-X^N(S_j(x)), \ j=1,\ldots,6,\] with $S_j$ given on Lemma \ref{si}, as the fixed choice for the representative of $\Xi_i$.

First of all, we consider complex coordinates $(z_1,z_2)\in\mathbb{C}^2$ instead of $(x_1,x_2,y_1,y_2)\in\rn{4}$:
\begin{equation}\label{cc}\left\{
\begin{array}{lcl}
z_1&=&x_1+ix_2\\
z_2&=&y_1+iy_2
\end{array}
\right.
\end{equation}

We will write $\Re(z)$ for the real part of the complex number $z$ and $\Im(z)$ for its imaginary part.

Define
\[\left\{\begin{array}{lcl}
\Delta_1&=&z_1\overline{z_1} \ \ \ (=x_1^2+x_2^2)\\
\Delta_2&=&z_2\overline{z_2} \ \ \ (=y_1^2+y_2^2)\\
\Delta_3&=&z_1^q\overline{z}_2^p\\
\Delta_4&=&\overline{\Delta}_3
\end{array}
\right.
\]

Note that each $\Delta_j$ corresponds to a relation represented by $$\Gamma_1^1 \lambda_1+ \Gamma_2^1\lambda_2 +\Gamma_1^2\overline{\lambda}_1 +\Gamma_2^2\overline{\lambda}_2=0, \textrm{ where } \Gamma_i^j\in\mathbb{N}.$$

It is not hard to see that the complex Belitskii normal form for $X$ in this case is expressed by
\begin{equation}\label{fncomplexa}\left\{
\begin{array}{lcl}
\dot{z}_1&=&piz_1+z_1f_1(\Delta_1,\Delta_2,\Delta_3,\Delta_4)+\overline{z}_1^{q-1}z_2^pf_2(\Delta_1,\Delta_2,\Delta_3,\Delta_4)\\
\dot{z}_2&=&qiz_2+z_2g_1(\Delta_1,\Delta_2,\Delta_3,\Delta_4)+z_1^{q}\overline{z}_2^{p-1}g_2(\Delta_1,\Delta_2,\Delta_3,\Delta_4),
\end{array}
\right.
\end{equation}
\noindent with $f_j,g_j$ without linear and constant terms.

Now we consider the effects of $\dd$-reversibility on the system \eqref{fncomplexa}. Writing our involutions in complex coordinates, we derive immediately that 
\begin{lema}\label{complexgroups}Let
\begin{equation*}
\begin{array}{lclclcl}
\varphi_0(z_1,z_2)&=&-(\overline{z_1},\overline{z_2})& \ \ \ & & &\\
\varphi_1(z_1,z_2)&=&(i\overline{z_1},\overline{z_2})& \ \ \  &\varphi_2(z_1,z_2)&=&-(\overline{z_1},\overline{iz_2})\\
\varphi_3(z_1,z_2)&=&(\overline{z_1},-\overline{iz_2})&  \ \ \ &\varphi_4(z_1,z_2)&=&-(\overline{iz_1},\overline{iz_2})\\
\varphi_5(z_1,z_2)&=&-(\overline{iz_1},\overline{z_2})&  \ \ \ &\varphi_6(z_1,z_2)&=&(-\overline{iz_1},\overline{iz_2})
\end{array}
\end{equation*}
Then each group $\langle \varphi_0, \varphi_j\rangle$ corresponds to $\langle R_0, S_j\rangle$, $j=1,\ldots,6$.
\end{lema}

To compute a $\dd$-reversible normal form for a vector field, one has first to define which of the groups in Lemma \ref{complexgroups} can be used to do the calculations. Now we establish a normal form of a $\dd$-reversible and $p:q$-resonant vector field $X$, depending only on $p,q$ and not on the involutions generating $\dd$:

\begin{teorema}\label{UT}Let $p,q$ be odd numbers with $pq>1$ and $X\in\mathfrak{X}_0^{(p,q)}(\rn{4})$ be a $\dd$-reversible vector field. Then $X$ is formally conjugated to the following system:
\begin{equation}\label{tnf}\left\{
\begin{array}{ccc}
\dot{x_1}&=&-px_{2}-x_2\sum_{i+j=1}^\infty a_{ij}\Delta_1^i\Delta_2^j\\
\dot{x_2}&=&px_{1}+x_1\sum_{i+j=1}^\infty a_{ij}\Delta_1^i\Delta_2^j\\
\dot{y_1}&=&-qy_{2}-y_2\sum_{i+j=1}^\infty b_{ij}\Delta_1^i\Delta_2^j\\
\dot{y_2}&=&qy_{1}+y_1\sum_{i+j=1}^\infty b_{ij}\Delta_1^i\Delta_2^j,
\end{array}
\right.
\end{equation}
\noindent with $a_{ij},b_{ij}\in\rn{}$ depending on $j^kX(0)$, for $k=i+j$.
\end{teorema}

\begin{obs}\label{remark10}The hypothesis on $p,q$ given in Theorem \ref{UT} can be relaxed. In fact, if $p,q$ satisfies the following conditions
\[\left\{
\begin{array}{l}
q\equiv_4 1 \ \textrm{or} \ q\equiv_4 3 \ \textrm{or} \ (q\equiv_4 0 \ \textrm{and} \ p+q=2k+1) \ \textrm{or} \ (q\equiv_4 2 \ \textrm{and} \ p+q=2k)\\
p\equiv_4 1 \ \textrm{or} \ p\equiv_4 2 \ \textrm{or} \ p\equiv_4 3\\
p\equiv_4 1 \ \textrm{or} \ p\equiv_4 3 \ \textrm{or} \ (p\equiv_4 0 \ \textrm{and} \ q=2k+1) \ \textrm{or} \ (p\equiv_4 2 \ \textrm{and} \ q=2k)
\end{array}
\right.
\]
\noindent then the conclusions of Theorem \ref{UT} are also valid (see \cite{martins}).
\end{obs}

\begin{obs}The normal form \eqref{tnf} coincides (in the nonlinear terms) with the normal form of a reversible vector field $X\in\mathfrak{X}_0^{(\alpha,\beta)}(\rn{4})$ with $\alpha\beta^{-1}\notin\mathbb{Q}$. Remember that this fact allowed us to discard the case $\alpha\beta^{-1}\notin\mathbb{Q}$ in page \pageref{opcoes}.
\end{obs}

The proof of Theorem \ref{UT} (even with the hypothesis of Remark \ref{remark10}) is based on a sequence of lemmas. The idea is just to show that with some hypothesis on $p$ and $q$, all the coefficients of $\Delta_3$ and $\Delta_4$ in the reversible-equivariant analogous of \eqref{fncomplexa} must zero.

First let us focus on the monomials that are never killed by the reversible-equivariant structure. 

\begin{lema}\label{mataocasogend4}Let $v=az_j\Delta_1^m\Delta_2^n\frac{\partial}{\partial z_j}$, $a\in\mathbb{C}$. So, for any $j\in\{1,\ldots,6\}$, the $\varphi_j$-reversibility implies $\overline{a}=-a$ (or $\Re(a)=0$). In particular, these terms are always present (generically) in the normal form.
\end{lema}
\begin{proof}From
\begin{equation*}
\varphi_0\left(az_1\Delta_1^m\Delta_2^n\frac{\partial}{\partial z_1}\right)=-\overline{a} \ \overline{z_1}\Delta_1^m\Delta_2^n\frac{\partial}{\partial z_1}
\end{equation*}
\noindent and
\begin{equation*}
\left.az_1\Delta_1^m\Delta_2^n\frac{\partial}{\partial z_1}\right|_{(-\overline{z_1},-\overline{z_2})}=-a\overline{z_1}\Delta_1^m\Delta_2^n\frac{\partial}{\partial z_1}
\end{equation*}
\noindent follows that $-\overline{a}=a$.
\end{proof}

Now let us see what happens with the monomials of type $(\overline{z_1})^{q-1}z_2^p\frac{\partial}{\partial z_1}$. We mention that only for such monomials a complete proof will be presented. The other cases are similar. Moreover, we will give the statement and the proof in the direction of Remark \ref{remark10}.

\begin{lema}\label{viva0}Let $v=b\overline{z_1}^{q-1}z_2^p\frac{\partial}{\partial z_1}$, $b\in\mathbb{C}$. So, we establish the following tables:

\[
\begin{array}{lll}
reversibility &\vline \  hypothesis \ on \ p,q &\vline \  conditions \ on \ b \\ \hline
\varphi_0   &\vline \ p+q \ even             &\vline \  \Re(b)=0           \\ \hline
            &\vline \ p+q \ odd              &\vline \  \Im(b)=0           \\ \hline
\varphi_1   &\vline \ q\equiv_4 0            &\vline \  \Re(b)=0           \\ \hline
            &\vline \ q\equiv_4 1            &\vline \  \Re(b)=-\Im(b)     \\ \hline
            &\vline \ q\equiv_4 2            &\vline \  \Im(b)=0           \\ \hline
            &\vline \ q\equiv_4 3            &\vline \  \Re(b)=\Im(b)      \\ \hline
\varphi_2   &\vline \ p\equiv_4 0, \ q \ even&\vline \  \Re(b)=0           \\ \hline
            &\vline \ p\equiv_4 0, \ q \ odd &\vline \  \Im(b)=0           \\ \hline
            &\vline \ q\equiv_4 1, \ q \ even&\vline \  \Re(b)=\Im(b)      \\ \hline
            &\vline \ q\equiv_4 1, \ q \ odd &\vline \  \Re(b)=-\Im(b)     \\ \hline
            &\vline \ q\equiv_4 2, \ q \ even&\vline \  \Im(b)=0           \\ \hline
            &\vline \ q\equiv_4 2, \ q \ odd &\vline \  \Re(b)=0           \\ \hline
            &\vline \ q\equiv_4 3, \ q \ even&\vline \  \Re(b)=-\Im(b)     \\ \hline
            &\vline \ q\equiv_4 3, \ q \ odd &\vline \  \Re(b)= \Im(b)     \\ \hline
\varphi_3   &\vline \ p\equiv_4 0            &\vline \  \Re(b)=0           \\ \hline
            &\vline \ p\equiv_4 1            &\vline \  \Re(b)=\Im(b)      \\ \hline
            &\vline \ p\equiv_4 2            &\vline \  \Im(b)=0           \\ \hline
            &\vline \ p\equiv_4 3            &\vline \  \Re(b)=-\Im(b)      \\ \hline
\end{array}
\]
\[
\begin{array}{lll}
reversibility &\vline \  hypothesis \ on \ p,q &\vline \  conditions \ on \ b \\ \hline
\varphi_4   &\vline \ p+q\equiv_4 0, \ q \ even  &\vline \  \Re(b)=0           \\ \hline
            &\vline \ p+q\equiv_4 0, \ q \ odd   &\vline \  \Im(b)=0           \\ \hline
            &\vline \ p+q\equiv_4 1, \ q \ even  &\vline \  \Re(b)=\Im(b)      \\ \hline
            &\vline \ p+q\equiv_4 1, \ q \ odd  &\vline \  \Re(b)=-\Im(b)     \\ \hline
            &\vline \ p+q\equiv_4 2, \ q \ even &\vline \  \Im(b)=0           \\ \hline
            &\vline \ p+q\equiv_4 2, \ q \ odd   &\vline \  \Re(b)=0           \\ \hline
            &\vline \ p+q\equiv_4 3, \ q \ even  &\vline \  \Re(b)=-\Im(b)     \\ \hline
            &\vline \ p+q\equiv_4 3, \ q \ odd   &\vline \  \Im(b)=\Im(b)      \\ \hline
\varphi_5   &\vline \ q\equiv_4 0, \ p+q \ even  &\vline \  \Re(b)=0           \\ \hline
            &\vline \ q\equiv_4 0, \ p+q \ odd   &\vline \  \Im(b)=0           \\ \hline
            &\vline \ q\equiv_4 1, \ p+q \ even  &\vline \  \Re(b)=\Im(b)      \\ \hline
            &\vline \ q\equiv_4 1, \ p+q \ odd   &\vline \  \Re(b)=-\Im(b)    \\ \hline
            &\vline \ q\equiv_4 2, \ p+q \ even  &\vline \  \Im(b)=0           \\ \hline
            &\vline \ q\equiv_4 2, \ p+q \ odd   &\vline \  \Re(b)=0           \\ \hline
            &\vline \ q\equiv_4 3, \ p+q \ even  &\vline \  \Re(b)=-\Im(b)     \\ \hline
            &\vline \ q\equiv_4 3, \ p+q \ odd   &\vline \  \Re(b)=\Im(b)      \\ \hline
\varphi_6   &\vline \ p+q\equiv_4 0              &\vline \  \Re(b)=0           \\ \hline
            &\vline \ p+q\equiv_4 1              &\vline \  \Re(b)=-\Im(b)      \\ \hline
            &\vline \ p+q\equiv_4 2              &\vline \  \Im(b)=0            \\ \hline
            &\vline \ p+q\equiv_4 3              &\vline \  \Re(b)= \Im(b)      \\ \hline
\end{array}
\]

\end{lema}

\begin{proof}Let us give the proof for $\varphi_2$-reversibility. The proof of any other case is similar. Note that
\begin{equation*}\left\{
\begin{array}{lcl}
\varphi_2(v(z_1,z_2))&=&-\overline{b}z_1^{q-1}\overline{z_2}^p\frac{\partial}{\partial z_1}\\
v(\varphi_2(z_1,z_2))&=&b(-1)^{q-1}i^pz_1^{q-1}\overline{z_2}^p
\end{array}
\right.
\end{equation*}
Then, from $\varphi_2(v(z))=-v(\varphi_2(z))$ we have $\overline{b}=(-1)^{q-1}i^pb$. Now we apply the hypotheses on $p,q$ and the proof follows in a straightforward way.
\end{proof}

Next corollary is the first of a sequence of results establishing that some monomial does not appear in the normal form:

\begin{corolario}\label{viva}Let $X\in\mathfrak{X}_0^{(p,q)}(\rn{4})$ be a $\langle\varphi_0,\varphi_j\rangle$-reversible vector field. Then if
\begin{itemize}
\item $q\equiv 1 \ mod \ 4$ or
\item $q\equiv 3 \ mod \ 4$ or 
\item $q\equiv 0 \ mod \ 4$ and $p+q$ odd or
\item $q\equiv 2 \ mod \ 4$ and $p+q$ even,
\end{itemize}
then the normal form of $X$ does not contain monomials of the form \begin{equation}\label{queira}a_1\overline{z_1}^{nq-1}z_2^np\frac{\partial}{\partial z_1}, \ \ a_2z_1^{mq}\overline{z}_2^{mp-1}\frac{\partial}{\partial z_2}, \ a_1,a_2\in\mathbb{C}.\end{equation}
\end{corolario}
\begin{proof}Observe that the $\varphi_0,\varphi_j$-reversibility implies that the coefficients in \eqref{queira} satisfy \[\Re(a_j)=\Im(a_j)=0.\]
\end{proof}
\begin{obs}Note that if $p,q$ are odd with $pq>1$, then they satisfy the hypothesis of Corollary \ref{viva}.
\end{obs}

The following results can be proved in a similar way as we have done in Lemma \ref{viva0} and Corollary \ref{viva}.
\begin{proposicao}\label{maior}Let $X\in\mathfrak{X}_0^{(p,q)}(\rn{4})$ be a $\langle\varphi_0,\varphi_j\rangle$-reversible vector field.  If one of the following conditions is satisfied:
\begin{itemize}
\item[(i)] $q\equiv 1 \ mod \ 4$,
\item[(ii)] $q\equiv 3 \ mod \ 4$,
\item[(iii)] $q\equiv 0 \ mod \ 4$ and $p+q$
\item[(iv)] $q\equiv 2 \ mod \ 4$ and $p+q$ even,
\end{itemize}
then the normal form of $X$, given in \eqref{fncomplexa}, does not have monomials of type \[z_1(z_1^q\overline{z}_2^p)^m\dfrac{\partial}{\partial z_1}, \ \ z_2(z_1^q\overline{z}_2^p)^m\dfrac{\partial}{\partial z_2}, \ \ m\geq 1.\]
\end{proposicao}
\begin{proposicao}\label{sucesso}Let $X\in\mathfrak{X}_0^{(p,q)}(\rn{4})$ be a $\langle\varphi_0,\varphi_j\rangle$-reversible vector field. If one of the following conditions is satisfied
\begin{itemize}
\item[(i)] $q\equiv 1 \ mod \ 4$,
\item[(ii)] $q\equiv 3 \ mod \ 4$,
\item[(iii)] $q\equiv 0 \ mod \ 4$ and $p+q$ odd,
\item[(iv)] $q\equiv 2 \ mod \ 4$ and $p+q$ even,
\end{itemize}
then the normal form of $X$, given in \eqref{fncomplexa}, does not have monomials of type \[z_1(\overline{z_1^q\overline{z}_2^p})^m\dfrac{\partial}{\partial z_1}, \ \  z_2(\overline{z_1^q\overline{z}_2^p})^m\dfrac{\partial}{\partial z_2}, \ \ m\geq 1.\]
\end{proposicao}

\begin{obs}In fact, the conditions imposed on $\lambda$ in the last results are needed just to assure the $\langle\varphi_0,\varphi_j\rangle$-reversibility of the vector field $X$ with $j=1$. For $2\leq j\leq 6$, the normal form only contains monomials of type $z_j\Delta_1^m\Delta_2^n\dfrac{\partial}{\partial z_j}$.
\end{obs}

Now, to prove Theorem \ref{UT}, we have just to combine all lemmas, corollaries and propositions given above.

\begin{proof} \textrm{(\textit{of Theorem \ref{UT}})} Note that the conditions imposed on $\lambda$ in Theorem \ref{UT} fit into the hypothesis of Corollary \ref{viva} and Propositions \ref{maior} and \ref{sucesso}. So, if $p,q$ are odd numbers with $pq>1$, then the normal form just have monomials of type \[z_j\Delta_1^m\Delta_2^n\dfrac{\partial}{\partial z_j}, \ j=1,2, \ m,n>1.\]
\end{proof}

\section{Conclusions}

We have classified all involutions that make a vector field $X\in\mathfrak{X}_0^{(p,q)}(\rn{4})$ $\langle\varphi,\psi\rangle$-reversible when the order of the group $\langle \varphi,\psi\rangle$ is smaller than $9$.

As a consequence of the results obtained in Theorem A, we find a normal form for $\dd$-reversible vector fields in $\rn{4}$, according to their resonances. In this part we have used some results from Normal Form Theory. 

We remark that the same approach can be made to the discrete version of the problem, or when the singularity is not elliptic (see for example \cite{alain}).

One can easily generalize the results presented here mainly in two directions: for vector fields on higher dimensional spaces and for groups with higher order. In both cases the hard missions are to face the normal form calculations and to solve some very complicate system of algebraic equations.

\section*{Appendix}
\addcontentsline{toc}{section}{Appendix}

Here we will show the polynomial system omitted in Section \ref{secaod4}, corresponding to the equations $SA+AS=0$, $S^2-Id=0$ and $SR_0-(R_0S)^3=0$, as in the page \pageref{ddd}.

The system is:

\tiny
\noindent ${ d_1}{ b_4}+{ c_1}{ b_3}+{ b_1}{ b_2}+{
 a_1}{ b_1}=0, \ \ \ \ \ { d_1}{ c_4}+{ c_1}{ c_3}+{ b_1}{ c_2}+{
 a_1}{ c_1}=0, \ \ \ \ \ 
{ d_1}{ d_4}+{ c_1}{ d_3}+{ b_1}{ d_2}+{
 a_1}{ d_1}=0, \ \ \ \ \ { d_2}{ a_4}+{ c_2}{ a_3}+{ a_2}{ b_2}+{
 a_1}{ a_2}=0\\
{ d_2}{ c_4}+{ c_2}{ c_3}+{ b_2}{ c_2}+{
 c_1}{ a_2}=0, \ \ \ \ \ { d_2}{ d_4}+{ c_2}{ d_3}+{ b_2}{ d_2}+{
 d_1}{ a_2}=0, \ \ \ \ \ 
{ d_3}{ a_4}+{ a_3}{ c_3}+{ a_2}{ b_3}+{
 a_1}{ a_3}=0, \ \ \ \ \ { d_3}{ b_4}+{ b_3}{ c_3}+{ b_2}{ b_3}+{
 b_1}{ a_3}=0\\
{ d_3}{ d_4}+{ c_3}{ d_3}+{ d_2}{ b_3}+{
 d_1}{ a_3}=0, \ \ \ \ \ { a_4}{ d_4}+{ a_3}{ c_4}+{ a_2}{ b_4}+{
 a_1}{ a_4}=0, \ \ \ \ \ 
{ b_4}{ d_4}+{ b_3}{ c_4}+{ b_2}{ b_4}+{
 b_1}{ a_4}=0, \ \ \ \ \ { c_4}{ d_4}+{ c_3}{ c_4}+{ c_2}{ b_4}+{
 c_1}{ a_4}=0\\
-\alpha{ a_2}+\alpha{ b_1}=0, \ \ \ \ \ -\alpha{ b_2}-\alpha{ a_1}=0, \ \ \ \ \ -\alpha{ c_2}+\beta{ d_1}=0, \ \ \ \ \ -\alpha{ d_2}-\beta{ c_1}=0, \ \ \ \ \ \alpha{ b_2}+\alpha{ a_1}=0, \ \ \ \ \ \beta{ d_2}+\alpha{ c_1}=0, \ \ \ \ \ -\beta{ c_2}+\alpha{ d_1}=0\\
-\beta{ a_4}+\alpha{ b_3}=0, \ \ \ \ \ -\beta{ b_4}-\alpha{ a_3}=0, \ \ \ \ \ -\beta{ c_4}+\beta{ d_3}=0,  \ \ \ \ \ -\beta{ d_4}-\beta{ c_3}=0, \ \ \ \ \ \alpha{ b_4}+\beta{ a_3}=0, \ \ \ \ \ -\alpha{ a_4}+\beta{ b_3}=0, \ \ \ \ \ \beta{ d_4}+\beta{ c_3}=0\\
{ d_1}{ a_4}+{ c_1}{ a_3}+{ b_1}{ a_2
}+{{ a_1}}^{2}=1, \ \ \ \ \ { d_2}{ b_4}+{ c_2}{ b_3}+{{ b_2}}^{2}+{
 b_1}{ a_2}=1, \ \ \ \ \  
{ d_3}{ c_4}+{{ c_3}}^{2}+{ c_2}{ b_3}+{
 c_1}{ a_3}=1, \ \ \ \ \ {{ d_4}}^{2}+{ d_3}{ c_4}+{ d_2}{ b_4}+{
 d_1}{ a_4}=1\\
{ a_1}-{ d_1}{ a_4}{ d_4}+{ d_1}{ a_3}
{ c_4}-{ d_1}{ a_2}{ b_4}+{ c_1}{ d_3}
{ a_4}-{ c_1}{ a_3}{ c_3}+{ c_1}{ a_2}
{ b_3}-{ b_1}{ d_2}{ a_4}+{ b_1}{ c_2}
{ a_3}-{ b_1}{ a_2}{ b_2}+2{ a_1}{ d_1
}{ a_4}-2{ a_1}{ c_1}{ a_3}+2{ a_1}{
 b_1}{ a_2}-{{ a_1}}^{3}=0\\
-{ b_1}-{ d_1}{ b_4}{ d_4}+{ d_1}{ b_3}
{ c_4}-{ d_1}{ b_2}{ b_4}+{ c_1}{ d_3}
{ b_4}-{ c_1}{ b_3}{ c_3}+{ c_1}{ b_2}
{ b_3}-{ b_1}{ d_2}{ b_4}+{ b_1}{ c_2}
{ b_3}-{ b_1}{{ b_2}}^{2}+{ b_1}{ d_1}{
 a_4}-{ b_1}{ c_1}{ a_3}+{{ b_1}}^{2}{ a_2
}+{ a_1}{ d_1}{ b_4}-{ a_1}{ c_1}{ b_3
}+{ a_1}{ b_1}{ b_2}-{{ a_1}}^{2}{ b_1}=0\\
{ c_1}-{ d_1}{ c_4}{ d_4}+{ d_1}{ c_3}
{ c_4}-{ d_1}{ c_2}{ b_4}+{ c_1}{ d_3}
{ c_4}-{ c_1}{{ c_3}}^{2}+{ c_1}{ c_2}{ 
b_3}+{ c_1}{ d_1}{ a_4}-{{ c_1}}^{2}{ a_3}-{
 b_1}{ d_2}{ c_4}+{ b_1}{ c_2}{ c_3}-{
 b_1}{ b_2}{ c_2}+{ b_1}{ c_1}{ a_2}+{
 a_1}{ d_1}{ c_4}-{ a_1}{ c_1}{ c_3}+{
 a_1}{ b_1}{ c_2}-{{ a_1}}^{2}{ c_1}=0\\
-{ d_1}-{ d_1}{{ d_4}}^{2}+{ d_1}{ d_3}{ 
c_4}-{ d_1}{ d_2}{ b_4}+{{ d_1}}^{2}{ a_4}+{
 c_1}{ d_3}{ d_4}-{ c_1}{ c_3}{ d_3}+{
 c_1}{ d_2}{ b_3}-{ c_1}{ d_1}{ a_3}-{
 b_1}{ d_2}{ d_4}+{ b_1}{ c_2}{ d_3}-{
 b_1}{ b_2}{ d_2}+{ b_1}{ d_1}{ a_2}+{
 a_1}{ d_1}{ d_4}-{ a_1}{ c_1}{ d_3}+{
 a_1}{ b_1}{ d_2}-{{ a_1}}^{2}{ d_1}=0\\
{ a_2}+{ d_2}{ a_4}{ d_4}-{ d_2}{ a_3}
{ c_4}-{ c_2}{ d_3}{ a_4}+{ c_2}{ a_3}
{ c_3}+{ b_2}{ d_2}{ a_4}-{ b_2}{ c_2}
{ a_3}+{ a_2}{ d_2}{ b_4}-{ a_2}{ c_2}
{ b_3}+{ a_2}{{ b_2}}^{2}-{ d_1}{ a_2}{ 
a_4}+{ c_1}{ a_2}{ a_3}-{ b_1}{{ a_2}}^{2}-
{ a_1}{ d_2}{ a_4}+{ a_1}{ c_2}{ a_3}-
{ a_1}{ a_2}{ b_2}+{{ a_1}}^{2}{ a_2}=0\\
-{ b_2}+{ d_2}{ b_4}{ d_4}-{ d_2}{ b_3}
{ c_4}-{ c_2}{ d_3}{ b_4}+{ c_2}{ b_3}
{ c_3}+2{ b_2}{ d_2}{ b_4}-2{ b_2}{
 c_2}{ b_3}+{{ b_2}}^{3}-{ d_1}{ a_2}{ 
b_4}+{ c_1}{ a_2}{ b_3}-{ b_1}{ d_2}{ 
a_4}+{ b_1}{ c_2}{ a_3}-2{ b_1}{ a_2}{
 b_2}+{ a_1}{ b_1}{ a_2}=0\\
{ c_2}+{ d_2}{ c_4}{ d_4}-{ d_2}{ c_3}
{ c_4}-{ c_2}{ d_3}{ c_4}+{ c_2}{{ c_3}}
^{2}+{ c_2}{ d_2}{ b_4}-{{ c_2}}^{2}{ b_3}+{
 b_2}{ d_2}{ c_4}-{ b_2}{ c_2}{ c_3}+{
{ b_2}}^{2}{ c_2}-{ d_1}{ a_2}{ c_4}-{ c_1
}{ d_2}{ a_4}+{ c_1}{ c_2}{ a_3}+{ c_1
}{ a_2}{ c_3}-{ c_1}{ a_2}{ b_2}-{ b_1
}{ a_2}{ c_2}+{ a_1}{ c_1}{ a_2}=0\\
-{ d_2}+{ d_2}{{ d_4}}^{2}-{ d_2}{ d_3}{ 
c_4}+{{ d_2}}^{2}{ b_4}-{ c_2}{ d_3}{ d_4}+{
 c_2}{ c_3}{ d_3}-{ c_2}{ d_2}{ b_3}+{
 b_2}{ d_2}{ d_4}-{ b_2}{ c_2}{ d_3}+{
{ b_2}}^{2}{ d_2}-{ d_1}{ d_2}{ a_4}+{ d_1
}{ c_2}{ a_3}-{ d_1}{ a_2}{ d_4}-{ d_1
}{ a_2}{ b_2}+{ c_1}{ a_2}{ d_3}-{ b_1
}{ a_2}{ d_2}+{ a_1}{ d_1}{ a_2}=0\\
{ a_3}-{ d_3}{ a_4}{ d_4}+{ c_3}{ d_3}
{ a_4}+{ a_3}{ d_3}{ c_4}-{ a_3}{{ c_3}}
^{2}-{ d_2}{ b_3}{ a_4}+{ c_2}{ a_3}{ 
b_3}-{ a_2}{ d_3}{ b_4}+{ a_2}{ b_3}{ 
c_3}-{ a_2}{ b_2}{ b_3}+{ d_1}{ a_3}{ 
a_4}-{ c_1}{{ a_3}}^{2}+{ b_1}{ a_2}{ a_3}+
{ a_1}{ d_3}{ a_4}-{ a_1}{ a_3}{ c_3}+
{ a_1}{ a_2}{ b_3}-{{ a_1}}^{2}{ a_3}=0\\
-{ b_3}-{ d_3}{ b_4}{ d_4}+{ c_3}{ d_3}
{ b_4}+{ b_3}{ d_3}{ c_4}-{ b_3}{{ c_3
}}^{2}-{ d_2}{ b_3}{ b_4}+{ c_2}{{ b_3}}^{2}
-{ b_2}{ d_3}{ b_4}+{ b_2}{ b_3}{ c_3}
-{{ b_2}}^{2}{ b_3}+{ d_1}{ a_3}{ b_4}-{ 
c_1}{ a_3}{ b_3}+{ b_1}{ d_3}{ a_4}-{ 
b_1}{ a_3}{ c_3}+{ b_1}{ b_2}{ a_3}+{ 
b_1}{ a_2}{ b_3}-{ a_1}{ b_1}{ a_3}=0\\
{ c_3}-{ d_3}{ c_4}{ d_4}+2{ c_3}{ d_3
}{ c_4}-{{ c_3}}^{3}-{ d_2}{ b_3}{ c_4}-{
 c_2}{ d_3}{ b_4}+2{ c_2}{ b_3}{ c_3
}-{ b_2}{ c_2}{ b_3}+{ d_1}{ a_3}{ c_4
}+{ c_1}{ d_3}{ a_4}-2{ c_1}{ a_3}{ 
c_3}+{ c_1}{ a_2}{ b_3}+{ b_1}{ c_2}{ 
a_3}-{ a_1}{ c_1}{ a_3}=0\\
-{ d_3}-{ d_3}{{ d_4}}^{2}+{{ d_3}}^{2}{ c_4}+{
 c_3}{ d_3}{ d_4}-{{ c_3}}^{2}{ d_3}-{ d_2
}{ d_3}{ b_4}-{ d_2}{ b_3}{ d_4}+{ d_2
}{ b_3}{ c_3}+{ c_2}{ b_3}{ d_3}-{ b_2
}{ d_2}{ b_3}+{ d_1}{ d_3}{ a_4}+{ d_1
}{ a_3}{ d_4}-{ d_1}{ a_3}{ c_3}+{ d_1
}{ a_2}{ b_3}-{ c_1}{ a_3}{ d_3}+{ b_1
}{ d_2}{ a_3}-{ a_1}{ d_1}{ a_3}=0\\
{ a_4}+{ a_4}{{ d_4}}^{2}-{ d_3}{ a_4}{ 
c_4}-{ a_3}{ c_4}{ d_4}+{ a_3}{ c_3}{ 
c_4}+{ d_2}{ a_4}{ b_4}-{ c_2}{ a_3}{ 
b_4}+{ a_2}{ b_4}{ d_4}-{ a_2}{ b_3}{ 
c_4}+{ a_2}{ b_2}{ b_4}-{ d_1}{{ a_4}}^{2}+
{ c_1}{ a_3}{ a_4}-{ b_1}{ a_2}{ a_4}-
{ a_1}{ a_4}{ d_4}+{ a_1}{ a_3}{ c_4}-
{ a_1}{ a_2}{ b_4}+{{ a_1}}^{2}{ a_4}=0\\
-{ b_4}+{ b_4}{{ d_4}}^{2}-{ d_3}{ b_4}{ 
c_4}-{ b_3}{ c_4}{ d_4}+{ b_3}{ c_3}{ 
c_4}+{ d_2}{{ b_4}}^{2}-{ c_2}{ b_3}{ b_4}+
{ b_2}{ b_4}{ d_4}-{ b_2}{ b_3}{ c_4}+
{{ b_2}}^{2}{ b_4}-{ d_1}{ a_4}{ b_4}+{ 
c_1}{ b_3}{ a_4}-{ b_1}{ a_4}{ d_4}+{ 
b_1}{ a_3}{ c_4}-{ b_1}{ b_2}{ a_4}-{ 
b_1}{ a_2}{ b_4}+{ a_1}{ b_1}{ a_4}=0\\
{ c_4}+{ c_4}{{ d_4}}^{2}-{ d_3}{{ c_4}}^{2}-{
 c_3}{ c_4}{ d_4}+{{ c_3}}^{2}{ c_4}+{ d_2
}{ b_4}{ c_4}+{ c_2}{ b_4}{ d_4}-{ c_2
}{ c_3}{ b_4}-{ c_2}{ b_3}{ c_4}+{ b_2
}{ c_2}{ b_4}-{ d_1}{ a_4}{ c_4}-{ c_1
}{ a_4}{ d_4}+{ c_1}{ c_3}{ a_4}+{ c_1
}{ a_3}{ c_4}-{ c_1}{ a_2}{ b_4}-{ b_1
}{ c_2}{ a_4}+{ a_1}{ c_1}{ a_4}=0\\
-{ d_4}+{{ d_4}}^{3}-2{ d_3}{ c_4}{ d_4}+{
 c_3}{ d_3}{ c_4}+2{ d_2}{ b_4}{ d_4
}-{ d_2}{ b_3}{ c_4}-{ c_2}{ d_3}{ b_4
}+{ b_2}{ d_2}{ b_4}-2{ d_1}{ a_4}{ 
d_4}+{ d_1}{ a_3}{ c_4}-{ d_1}{ a_2}{ 
b_4}+{ c_1}{ d_3}{ a_4}-{ b_1}{ d_2}{ 
a_4}+{ a_1}{ d_1}{ a_4}=0\\
{ a_1}-{ d_1}{ a_4}{ d_4}+{ d_1}{ a_3}
{ c_4}-{ d_1}{ a_2}{ b_4}+{ c_1}{ d_3}
{ a_4}-{ c_1}{ a_3}{ c_3}+{ c_1}{ a_2}
{ b_3}-{ b_1}{ d_2}{ a_4}+{ b_1}{ c_2}
{ a_3}-{ b_1}{ a_2}{ b_2}+2{ a_1}{ d_1
}{ a_4}-2{ a_1}{ c_1}{ a_3}+2{ a_1}{
 b_1}{ a_2}-{{ a_1}}^{3}=0\\
-{ b_1}-{ d_1}{ b_4}{ d_4}+{ d_1}{ b_3}
{ c_4}-{ d_1}{ b_2}{ b_4}+{ c_1}{ d_3}
{ b_4}-{ c_1}{ b_3}{ c_3}+{ c_1}{ b_2}
{ b_3}-{ b_1}{ d_2}{ b_4}+{ b_1}{ c_2}
{ b_3}-{ b_1}{{ b_2}}^{2}+{ b_1}{ d_1}{
 a_4}-{ b_1}{ c_1}{ a_3}+{{ b_1}}^{2}{ a_2
}+{ a_1}{ d_1}{ b_4}-{ a_1}{ c_1}{ b_3
}+{ a_1}{ b_1}{ b_2}-{{ a_1}}^{2}{ b_1}=0\\$
\normalsize

The solutions of the system above were given in the Lemma \ref{solmonstro}.

\end{document}